\documentclass{amsart}
\usepackage{amssymb}
\usepackage{amscd}
\usepackage[latin1]{inputenc}
\usepackage{amsmath}
\usepackage{amsfonts}
\usepackage{latexsym}
\usepackage{verbatim}
\usepackage{graphicx}
\usepackage{epsfig}
\usepackage{color}
\usepackage{caption}
\newtheorem{theorem}{Theorem}[section]

\newtheorem{lemma}[theorem]{Lemma}

\theoremstyle{definition}
\newtheorem{definition}[theorem]{Definition}
\newtheorem{example}[theorem]{Example}
\newtheorem{remark}[theorem]{Remark}

\numberwithin{equation}{section}

\begin{document}
\title[Characterizing Follower and Extender Set Sequences]{Characterizing Follower and Extender Set Sequences}

\begin{abstract}

Given a one-dimensional shift $X$, let $|F_X(\ell)|$ be the number of follower sets of words of length $\ell$ in $X$. We call the sequence $\{|F_X(\ell)|\}_{\ell \in \mathbb{N}}$ the follower set sequence of the shift $X$. Extender sets are a generalization of follower sets (see ~\cite{KassMadden}), and we define the extender set sequence similarly. In this paper, we explore which sequences may be realized as follower set sequences and extender set sequences of one-dimensional sofic shifts. We show that any follower set sequence or extender set sequence of a sofic shift must be eventually periodic. We also show that, subject to a few constraints, a wide class of eventually periodic sequences are possible. In fact, any natural number difference in the $\limsup$ and $\liminf$ of these sequences may be achieved, so long as the $\liminf$ of the sequence is sufficiently large.
\end{abstract}

\date{}
\author{Thomas French}
\address{Thomas French\\
Department of Mathematics\\
University of Denver\\
2280 S. Vine St.\\
Denver, CO 80208}
\email{Thomas.French@du.edu}

%\thanks{}
%\keywords{$\mathbb{Z}^d$; shift of finite type; sofic; multidimensional}
%\renewcommand{\subjclassname}{MSC 2010}
%below are the definitions for some subject classification numbers
%22D40 Ergodic theory on groups 
%37A05 Measure-preserving transformations 
%37A15 General groups of measure preserving transformations 
%37A35 Entropy and other invariants, isomorphism, classification 
%37B10 symbolic dynamics
%37B40 topological entropy
%37B50 multi-dimensional shifts of finite type, tiling systems 
%37C40 smooth ergodic theory, invariant measures
%37C45 dimension theory of dynamical systems
%37C85 Dynamics of group actions other than Z and R, and foliations 
%37C99 smooth dynamical systems, general theory
%37D35 thermodynamic formalism, variational principles, equilibrium states
\maketitle
%\tableofcontents

\section{Introduction}
\label{intro}

The complexity function of a shift $X$, $\Phi_X(\ell)$, counts the number of words of a given length $\ell$ in the language of the shift $X$. This function is natural to study; in particular, it may be used to calculate topological entropy of symbolic shifts. The Morse-Hedlund Theorem implies that if there exists $\ell \in \mathbb{N}$ with $\Phi_X(\ell) \leq \ell$, then every sequence in $X$ must be periodic. (See ~\cite{MorseHedlund}) \newline 
\indent For any $\mathbb{Z}$ shift $X$ and finite word $w$ appearing in some point of $X$, the \textbf{follower set} of $w$, written $F_X(w)$, is defined as the set of all one-sided infinite sequences $s$ such that the infinite word $ws$ occurs in some point of $X$. The \textbf{extender set} of $w$, written $E_X(w)$, is the set of all pairs of infinite sequences $(s,u)$, $s$ left-infinite and $u$ right-infinite, such that $swu$ is a point of $X$. It is well-known that for a $\mathbb{Z}$ shift $X$, finiteness of $\{F_X(w) \ | \ w \text{ in the language of $X$} \}$ is equivalent to $X$ being sofic, that is, the image of a shift of finite type under a continuous shift-commuting map. (see ~\cite{LindMarcus}) (In fact it is true that finiteness of $\{E_X(w) \ | \ w \text{ in the language of $X$}\}$ is equivalent to $X$ being sofic as well, see ~\cite{OrmesPavlov}).

We define the set $F_X(\ell)$ to be $\{F_X(w) \ | \ w \text{ has length }\ell\}$ for any positive integer $\ell$. Thus $|F_X(\ell)|$ is the total number of follower sets which correspond to some word $w$ of length $\ell$ in $X$. $|E_X(\ell)|$ is defined similarly for extender sets. Since the alphabet is finite, there are only finitely many words of a given length $\ell$, and so for any shift (sofic or not), $|F_X(\ell)|$ and $|E_X(\ell)|$ are finite for every $\ell$. If $X$ is sofic, $\{F_X(w) \ | \ w \text{ in the language of $X$} \}$ is finite, and thus the follower set sequence $\{|F_X(\ell)|\}_{\ell \in \mathbb{N}}$ must be bounded, and similarly for the extender set sequence $\{|E_X(\ell)|\}_{\ell \in \mathbb{N}}$. Ormes and Pavlov have proved a result in the style of Morse-Hedlund, that if $|E_X(\ell)| \leq \ell$ for any $\ell \in \mathbb{N}$, then $X$ is necessarily sofic. (See ~\cite{OrmesPavlov}) This may lead one to believe that there is a connection between those sequences which may appear as complexity sequences of shifts and those which may appear as extender set sequences of shifts. The results in this paper suggest otherwise; in particular, extender set sequences need not be monotonic!

First, any follower set sequence or extender set sequence of a one-dimensional sofic shift must be eventually periodic:

\begin{theorem}\label{eventuallyperiodic}
Let $X$ be a one-dimensional sofic shift, $p$ be one greater than the total number of extender sets in $X$, and $p_0$ be one greater than the total number of follower sets in $X$. Then the extender set sequence $\{|E_X(\ell)|\}_{\ell \in \mathbb{N}}$ is eventually periodic, where the periodicity must begin before the $p(1+p!)^{th}$ term, and the least eventual period is at most $p!$. The follower set sequence $\{|F_X(\ell)|\}_{\ell \in \mathbb{N}}$ is eventually periodic, where the periodicity must begin before the $p_0(1+p_0!)^{th}$ term, and the least eventual period is at most $p_0!$.
\end{theorem}

For simple examples, the follower set sequence and extender set sequence of a sofic shift are eventually constant, but in fact, sequences may be realized which are merely eventually periodic. This is in contrast to complexity sequences of shifts, which must be nondecreasing.  Martin Delacourt discovered the first such example in 2013. (see page 8 of ~\cite{OrmesPavlov}) In fact, a wide class of eventually periodic sequences may be realized:
 
\begin{theorem}\label{mainthm}
Let $n \in \mathbb{N}$, and $\mathcal{A} = \{A_1, A_2, A_3, ..., A_k\}$ be a nontrivial partition of $\{0, 1, ..., n-1\}$. Let $0 = r_1 < r_2< ... < r_k$ be natural numbers. Then there exists $m \in \mathbb{N}$ and an irreducible graph $\mathcal{G}$ such that the number of follower sets in $X_\mathcal{G}$ of words of length $\ell$ where $\ell \geq n+2$ and $\ell \pmod n \in A_j$ will be exactly $m + r_j$ for all $1 \leq j \leq k$. Furthermore, $m$ may be chosen such that $m < (6n+3)r_k$.
\end{theorem}

We also prove a similar result for extender sets:

\begin{theorem}\label{ESmain}
Let $n \in \mathbb{N}$, and $\mathcal{A} = \{A_1, A_2, A_3, ..., A_k\}$ be a nontrivial partition of $\{0, 1, ..., n-1\}$. Let $0 = r_1 < r_2< ... < r_k$ be natural numbers. Then there exists $m \in \mathbb{N}$ and an irreducible graph $\mathcal{G}$ such that the number of extender sets in $X_\mathcal{G}$ of words of length $\ell$ where $\ell \geq 14r_kn - 1$ and $\ell \pmod n \in A_j$ will be exactly $m + r_j$ for all $1 \leq j \leq k$. Furthermore, $m$ may be chosen such that $m \leq 39n^2r_k^2$.
\end{theorem}

The proofs of Theorems~\ref{mainthm} and~\ref{ESmain} are broken into two parts. First, we define a process, given $n \in \mathbb{N}$ and $S \subset \{0, 1, ... , n-1\}$, of constructing a graph $\mathcal{G}_{n,S}$ which gives words of length $\ell \pmod n \in S$ one greater follower and extender set than words of length $\ell \pmod n \notin S$ whenever $\ell$ is sufficiently large. We then describe a method of combining these graphs which results in a new shift, where for each $\ell$, the number of follower or extender sets of words of length $\ell$ is the sum of the number of follower or extender sets of words of length $\ell$ in each of the original shifts, plus a constant which does not depend on $\ell$. Combining these two propositions proves the result.

Finally, while non-sofic shifts always have follower set sequences and extender set sequences which go to infinity, we show that they need not do so in a monotone increasing fashion:

\begin{theorem}\label{Non-sofic}
There exists an irreducible non-sofic shift $X$ such that $\{|F_X(\ell)|\}_{\ell \in \mathbb{N}}$ and $\{|E_X(\ell)|\}_{\ell \in \mathbb{N}}$ are not monotone increasing.
\end{theorem}

\section{Definitions and preliminaries}
\label{defns}
Let $A$ denote a finite set, which we will refer to as our alphabet. 

\begin{definition}
A \textbf{word} over $A$ is a member of $A^n$ for some $n \in \mathbb{N}$.
\end{definition}

\begin{definition} For any words $v \in A^n$ and $w \in A^m$, we define the \textbf{concatenation} $vw$ to be the pattern in $A^{n+m}$ whose first $n$ letters are the letters forming $v$ and whose next $m$ letters are the letters forming $w$.
\end{definition}

\begin{definition} 
The \textbf{language} of a $\mathbb{Z}$ shift $X$, denoted by $L(X)$, is the set of all words which appear in points of $X$. For any finite $\ell \in \mathbb{N}$, $L_\ell(X) := L(X) \cap A^\ell$, the set of words in the language of $X$ with length $\ell$.
\end{definition}

\begin{definition}
For any one-dimensional shift $X$ over the alphabet $A$, and any word $w$ in the language of $X$, we define the \textbf{follower set of w in $X$}, $F_X(w)$, to be the set of all one-sided infinite sequences $s \in A^\mathbb{N}$ such that the infinite word $ws$ occurs in some point of $X$.
\end{definition}

\begin{definition}
For any one-dimensional shift $X$ over the alphabet $A$, and any word $w$ in the language of $X$, we define the \textbf{extender set of w in $X$}, $E_X(w)$, to be the set of all pairs $(s, u)$ where $s$ is a left-infinite sequence of symbols in $A$, $u$ is a right-infinite sequence of symbols in $A$, and $swu$ is a point of $X$.
\end{definition}

\begin{remark}
For any word $w \in L(X)$, define a projection function $f_w:E_X(w) \rightarrow F_X(w)$ by $f(s,u) = u$. Such a function sends the extender set of $w$ onto the follower set of $w$. Any two words $w, v$ with the same extender set would have the property then that $f_w(E_X(w))= f_v(E_X(v))$, that is, that $w$ and $v$ have the same follower set.
\end{remark}

\begin{definition}
For any positive integer $\ell$, define the set $F_X(\ell) = \{F_X(w) \ | \ w \in L_{\ell}(X)\}$. Thus the cardinality $|F_X(\ell)|$ is the number of distinct follower sets of words of length $\ell$ in $X$. Similarly, define $E_X(\ell) = \{E_X(w) \ | \ w \in L_\ell(X)\}$, so that $|E_X(\ell)|$ is the number of distinct extender sets of words of length $\ell$ in $X$.
\end{definition}

\begin{definition}
Given a shift $X$, the \textbf{follower set sequence of $X$} is the sequence $\{|F_X(\ell)|\}_{\ell \in \mathbb{N}}$. The \textbf{extender set sequence of $X$} is the sequence $\{|E_X(\ell)|\}_{\ell \in \mathbb{N}}$.
\end{definition}

\begin{example}
Let $X$ be a full shift on the alphabet $A$. Then any word $w \in L(X)$ may be followed legally by any sequence in $A^{\mathbb{N}}$, and thus the follower set of any word is the same. Hence there is only one follower set in a full shift. Similarly, there is only one extender set in a full shift. Then $\{|F_X(\ell)|\}_{\ell \in \mathbb{N}} = \{|E_X(\ell)|\}_{\ell \in \mathbb{N}} = \{1, 1, 1, ...\}$.
\end{example}

\begin{example}
The even shift is the one-dimensional sofic shift with alphabet $\{0,1\}$ defined by forbidding odd runs of zeros between ones. It is a simple exercise to show that the even shift has three follower sets, $F(0), F(1),$ and $F(10)$. The follower set sequence of the even shift is $\{|F_X(\ell)|\}_{\ell \in \mathbb{N}} = \{2, 3, 3, 3, ...\}$. It is easy to verify that for any word $w$ in the language of the even shift, the follower set of $w$ is identical to the follower set of one of these three words.
\end{example}

\begin{definition}
A one-dimensional shift $X$ is \textbf{sofic} if it is the image of a shift of finite type under a continuous shift-commuting map. 
\end{definition}

Equivalently, a shift $X$ is sofic iff there exists a finite directed graph $\mathcal{G}$ with labeled edges such that the points of $X$ are exactly the sets of labels of bi-infinite walks on $\mathcal{G}$. Then $\mathcal{G}$ is a \textbf{presentation} of $X$ and we say $X = X_\mathcal{G}$ (or that $X$ is the \textbf{edge shift} presented by $\mathcal{G}$). Another well-known equivalence is that sofic shifts are those with only finitely many follower sets, that is, a shift $X$ is sofic iff $\{F_X(w) \ | \ w \text{ in the language of $X$} \}$ is finite. The same equivalence exists for extender sets: X is sofic iff $\{E_X(w) \ | \ w \text{ in the language of $X$} \}$ is finite. (see ~\cite{OrmesPavlov}) This necessarily implies that for a sofic shift $X$, the follower set sequence and extender set sequence of $X$ are bounded. In fact, the converse is also true: if the follower set or extender set sequence of a shift $X$ is bounded, then $X$ is necessarily sofic. (See ~\cite{OrmesPavlov})

\begin{definition}
A directed labeled graph $\mathcal{G}$ is \textbf{irreducible} if for every ordered pair $(I, J)$ of vertices in $\mathcal{G}$, there exists a path in $\mathcal{G}$ from $I$ to $J$. 
\end{definition}

Results about shifts presented by graphs which are not irreducible may often be found by considering the reducible graph's irreducible components; for this reason, results in Sections ~\ref{followers} and ~\ref{XcrossY} of this paper will focus on the irreducible case.

\begin{definition}
A directed labeled graph $\mathcal{G}$ is \textbf{primitive} if $\exists \> N \in \mathbb{N}$ such that for every $n \geq N$, for every ordered pair $(I, J)$ of vertices in $\mathcal{G}$, there exists a path in $\mathcal{G}$ from $I$ to $J$ of length $n$. The least such $N$ is the \textbf{primitivity distance} for $\mathcal{G}$.
\end{definition}

\begin{definition}
A directed labeled graph $\mathcal{G}$ is \textbf{right-resolving} if for each vertex $I$ of $\mathcal{G}$, all edges starting at $I$ carry different labels. Similarly, $\mathcal{G}$ is \textbf{left-resolving} if for each vertex $I$ of $\mathcal{G}$, all edges ending at $I$ carry different labels.
\end{definition}

\begin{definition}
A directed labeled graph $\mathcal{G}$ is \textbf{follower-separated} if distinct vertices in $\mathcal{G}$ correspond to distinct follower sets--for all vertices $I, J$ in $\mathcal{G}$, there exists a one-sided infinite sequence $s$ of labels which may follow one vertex but not the other. %The graph $\mathcal{G}$ is \textbf{strongly follower-separated} if distinct sets of vertices in $\mathcal{G}$ correspond to distinct follower sets. Then, for any two words $w$ and $v$ in $L(X_\mathcal{G})$, $w$ and $v$ will only have the same follower set if paths labeled $w$ and paths labeled $v$ end at exactly the same set of vertices.
\end{definition}

\begin{definition}
A directed labeled graph $\mathcal{G}$ is \textbf{extender-separated} if distinct pairs of vertices correspond to distinct extender sets--for any two distinct pairs of initial and terminal vertices $\{I \rightarrow I'\}$ and $\{J \rightarrow J'\}$ such that there exist paths in $\mathcal{G}$ from $I$ to $I'$ and from $J$ to $J'$, there exists some word $w$ which is the label of a path in $\mathcal{G}$ beginning and ending with one pair of vertices, and pair $(s,u)$, $s$ a left-infinite sequence, $u$ a right-infinite sequence, such that $swu$ is a point of $X_\mathcal{G}$, but for every word $v$ which is the label of some path beginning and ending with the other pair of vertices, $svu$ is not a point of $X_\mathcal{G}$.% A directed labeled graph $\mathcal{G}$ is \textbf{strongly extender-separated} if distinct sets of pairs of vertices correspond to distinct extender sets. Then, for any two words $w$ and $v$ in $L(X_\mathcal{G})$, $w$ and $v$ will only have the same extender set if paths labeled $w$ and paths labeled $v$ have exactly the same set of pairs of initial and terminal vertices.
\end{definition}

\begin{definition} 
Given a directed labeled graph $\mathcal{G}$, a word $w$ is \textbf{right-synchronizing} if all paths in $\mathcal{G}$ labeled $w$ terminate at the same vertex. The word $w$ is \textbf{left-synchronizing} if all paths in $\mathcal{G}$ labeled $w$ begin at the same vertex. The word $w$ is \textbf{bi-synchronizing} if $w$ is both left- and right-synchronizing. A \textbf{bi-synchronizing letter} is a bi-synchronizing word of length 1.
\end{definition}

In fact, every one-dimensional sofic shift has a presentation $\mathcal{G}$ which is right-resolving, follower-separated, and contains a right-synchronizing word. (See ~\cite{LindMarcus})

\section{Eventual Periodicity of Follower and Extender Set Sequences}
\label{periodicity}

First, we show that only eventually periodic sequences may appear as follower set sequences  or extender set sequences of one-dimensional sofic shifts. To establish this result, we first need a lemma which is reminiscent of the pumping lemma. (See ~\cite{Lawson})

\begin{lemma}\label{pumping}
Let $X$ be a sofic shift, define $p$ to be one greater than the total number of extender sets in $X$, and define $p_0$ to be one greater than the total number of follower sets in $X$. Then all words $w$ in $L(X)$ of length $\ell \geq p $ may be written as $w = xyz$, where $|y| \geq 1$ and the word $xy^iz$ has the same extender set as the word $w$ for all $i \in \mathbb{N}$. Furthermore, all words $w$ in $L(X)$ of length $\ell \geq p_0$ may be written as $w = xyz$, where $|y| \geq 1$ and the word $xy^iz$ has the same follower set as the word $w$ for all $i \in \mathbb{N}$.
\end{lemma}

\begin{proof}
Since $X$ is sofic, $X$ has only finitely many extender sets. Let $p$ be one greater than the number of extender sets in $X$. Since $X$ is represented by a finite labeled graph $\mathcal{G}$, and there are only $|V(\mathcal{G})|^2$ possible pairs of vertices in $\mathcal{G}$, and each extender set must correspond to a non-empty set of pairs, we have that $p \leq 2^{(|V(\mathcal{G})|^2)}$. Let $w$ be a word in $L(X)$ of length $\ell \geq p$. Consider the prefixes of $w$. Since $w$ is of length at least $p$, there exist two prefixes of $w$ (one necessarily a strict subword of the other) with the same extender set, say $x$ and $xy$, where $|y| \geq 1$. Then for any pair $(s, u)$ of infinite sequences, $sxu$ is a point of $X$ if and only if $sxyu$ is a point of $X$ also. Call the remaining portion of $w$ by $z$ so that $w = xyz$. (We may have $|z| = 0$).\newline
\indent Now, let $(s, u)$ be in the extender set of $w$, that is, that $swu$ is a point of $X$. But $swu = sxyzu$, so $(s, yzu)$ is in the extender set of $x$. By above, then, $(s, yzu)$ is in the extender set of $xy$ also, that is, that $sxyyzu$ is a point of $X$. Hence, $(s, u)$ is in the extender set of $xyyz$. So $E_X(xyz) \subseteq E_X(xyyz)$. On the other hand, if $(s, u)$ is in the extender set of $xyyz$, then $sxyyzu$ is a point of $X$, and so $(s, yzu)$ is in the extender set of $xy$, and therefore the extender set of $x$. Thus $sxyzu$ is a point of $X$, and so $(s, u)$ is in the extender set of $xyz = w$. Therefore $E_X(xyz) = E_X(xyyz)$. Applying this argument repeatedly gives that $E_X(xyz) = E_X(xy^iz)$ for any $i \in \mathbb{N}$. \newline
\indent Letting $p_0$ be one greater than the total number of follower sets in $X$, and identical argument gives the corresponding result for follower sets.
\end{proof}

We use this lemma to establish Theorem~\ref{eventuallyperiodic}, that all follower set and extender set sequences of one-dimensional sofic shifts must be eventually periodic:

\begin{proof}[Proof of Theorem~\ref{eventuallyperiodic}]
Let $X$ be a one-dimensional sofic shift. We prove that the sequences $\{F_X(\ell)\}$ and $\{E_X(\ell)\}$ (that is, the sequences which record not only the number of follower sets of each length, but also the identities of those sets) are eventually periodic, which will trivially imply eventual periodicity of $\{|F_X(\ell)|\}$ and $\{|E_X(\ell)|\}$. \newline
\indent Let $w$ be a word of length $p$ in $L(X)$, where $p$ is defined as in Lemma~\ref{pumping}. Then by Lemma~\ref{pumping}, $w = xy_wz$ where $|y_w| \geq 1$ and $xy_w^iz$ has the same extender set as $w$ for all $i \in \mathbb{N}$. Let $k = lcm\{|y_w|\> | \> w \in L_p(X)\}$. Since the longest such a word $y_w$ could be is $p$, we have $k \leq p!$. Clearly for any word $w \in L_p(X)$, there is an $i \in \mathbb{N}$ such that $xy_w^iz \in L_{p+k}(X)$. Therefore, every extender set in $E_X(p)$ is also an extender set in $E_X(p+k)$, so $E_X(p) \subseteq E_X(p+k)$. \newline
\indent Now, let $w$ be a word of length $\ell > p$ in $L(X)$. Then $w$ has some word $w' = w_1...w_p \in L_p(X)$ as a prefix. Applying Lemma ~\ref{pumping} to $w'$ as above, we get a word $w''$ of length $p+k$ with the same extender set as $w'$. If $(s,u)$ is in the extender set of $w = w'w_{p+1}...w_\ell$, then $sw'w_{p+1}...w_\ell u$ is a point of $X$, and so $(s, w_{p+1}...w_\ell u)$ is in the extender set of $w'$. Since $w'$ and $w''$ have the same extender set, $sw''w_{p+1}...w_\ell u$ is a point of $X$, and $(s,u)$ is in the extender set of $w''w_{p+1}...w_\ell$. Similarly, if $(s,u)$ is in the extender set of $w''w_{p+1}...w_\ell$, then $(s,u)$ is in the extender set of $w$ as well. Therefore $w''w_{p+1}...w_\ell$ is a word in $X$ of length $\ell + k$ with the same extender set as $w$. Hence, every extender set in $E_X(\ell)$ is an extender set in $E_X(\ell + k)$. So $E_X(\ell) \subseteq E_X(\ell + k)$ for any $\ell \geq p$. \newline
\indent But sofic shifts only have finitely many extender sets, so eventually, the sequence $\{|E_X(\ell + jk)|\}_{j \in \mathbb{N}}$ must stop growing. Thus, we have $E_X(\ell) = E_X(\ell+k)$ for all sufficiently large $\ell$, and the sequence $\{E_X(\ell)\}$ is eventually periodic with period $k$, where $k \leq p!$. Certainly, this implies that the extender set sequence is eventually periodic with period $k$ as well. \newline
\indent Suppose $\ell \geq p$ and $E_X(\ell) = E_X(\ell+k)$. Then we claim that $E_X(\ell+k) = E_X(\ell+2k)$: By above, we have $E_X(\ell+k) \subseteq E_X(\ell+2k)$. Suppose $w$ is a word of length $\ell + 2k$, say $w = w_1w_2...w_{\ell+k}w_{\ell+k+1}...w_{\ell+2k}$. Then, because $E_X(\ell) = E_X(\ell+k)$, there exists a word $z$ of length $\ell$ such that $z$ and $w_1w_2...w_{\ell+k}$ have the same extender set. Let $(s,u)$ be in the extender set of $w$. Then $sw_1w_2...w_{\ell+k}w_{\ell+k+1}...w_{\ell+2k}u$ is a point of $X$, so $(s,w_{\ell+k+1}...w_{\ell+2k}u)$ is in the extender set of $w_1w_2...w_{\ell+k}$, and thus in the extender set of $z$. So $szw_{\ell + k + 1}w_{\ell + k + 2}...w_{\ell+2k}u$ is a point of X. Hence $(s, u)$ is in the extender set of $zw_{\ell + k + 1}w_{\ell + k + 2}...w_{\ell+2k}$, a word of length $\ell +k$. Similarly, if $(s,u)$ is in the extender set of $zw_{\ell + k + 1}w_{\ell + k + 2}...w_{\ell+2k}$, then $(s,u)$ is in the extender set of $w$, giving $E_X(w) = E_X(zw_{\ell + k + 1}w_{\ell + k + 2}...w_{\ell+2k})$. Therefore we have $E_X(\ell+2k) \subseteq E_X(\ell+k)$ and we may conclude that $E_X(\ell+k) = E_X(\ell+2k)$. \newline \indent
Now,  $\{|E_X(\ell)|\} < p$ for any given $\ell \in \mathbb{N}$. Moreover, we have proven that the sequence $\{E_X(\ell + jk)\}_{j \in \mathbb{N}}$ is nondecreasing and nested by inclusion, and once two terms of the sequence are equal, it will stabilize for all larger $j$. The sequence $\{E_X(\ell + jk)\}_{j \in \mathbb{N}}$ must grow fewer than $p$ times, so the periodicity of the sequence $\{E_X(\ell)\}$ (and thus of $\{|E_X(\ell)|\}$) must begin before the $p + pk^{th}$ term, and $p + pk \leq p + p(p!) = p(1 + p!)$. \newline
\indent Again, a similar argument using the follower set portion of Lemma ~\ref{pumping} establishes the corresponding result for follower sets.
\end{proof}

\section{Existence of Eventually Periodic Nonconstant Follower and Extender Set Sequences}
\label{followers}

Now we demonstrate the existence of sofic shifts with follower set sequences which are not eventually constant. Given $n \in \mathbb{N}$ and $S \subset \{0, 1, ..., n-1\}$, construct an irreducible graph $\mathcal{G}_{n,S}$ in the following way: \newline

First, place edges labeled $p, q$ and $b$ as below, followed by a loop of $n$-many edges labeled $a$. We will refer to the initial vertex of the edge $p$ as ``Start." \newline

\includegraphics[scale=.7]{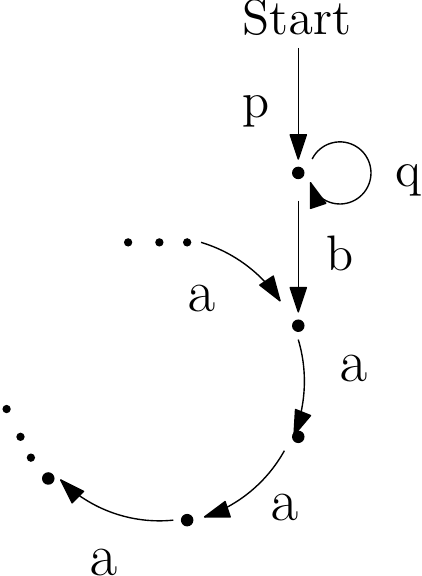}

Choose a fixed $i^* \in S$. Then for every $i \in S, \> i \neq i^*$, place two consecutive edges $c_i$ and $d_i$, such that the intial vertex of $c_i$ is the $(i-2 \pmod n)^{th}$ vertex of the loop of edges labeled $a$ (where we make the convention that the terminal vertex of the edge $b$ is the $0^{th}$ vertex of the loop, the next vertex the $1^{st}$ vertex of the loop, and so on) and the terminal vertex of $d_i$ is ``Start." For $i^*$, we still add the edge $c_{i^*}$, but follow it instead by another edge labeled $b$ and another loop of $n$-many edges labeled $a$: \newline \newline

\includegraphics[scale=.7]{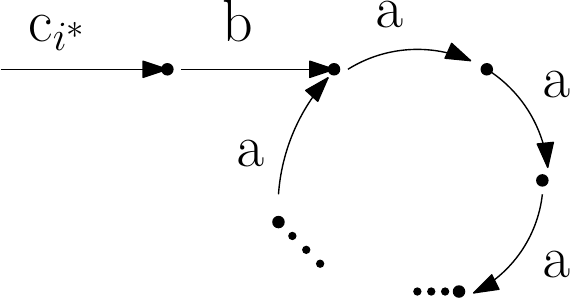}

Again, $\forall \> i \in S$, at the $(i - 2 \pmod n)^{th}$ vertex of the new loop, add an edge $c_i$, and if $i \neq i^*$, follow with an edge $e_i$ returning to Start. After $c_{i^*}$, add a third loop of $n$-many edges labeled $a$: \newline

\includegraphics[scale=.7]{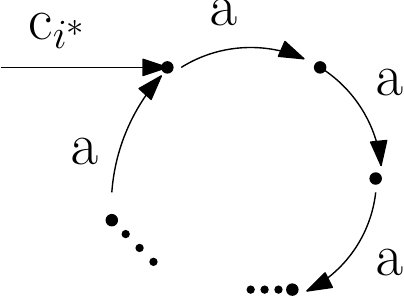}

Finally, for each $i \in S$, (including $i^*$), add an edge $c_i$ from the $(i - 2 \pmod n)^{th}$ vertex of the third loop returning to Start. The resulting graph is $\mathcal{G}_{n,S}$. Due to choice of $i^*$, there are $|S|$-many possible graphs $\mathcal{G}_{n,S}$; the results of this paper will hold for any of them.

\begin{figure}[h]
\includegraphics[scale=0.7]{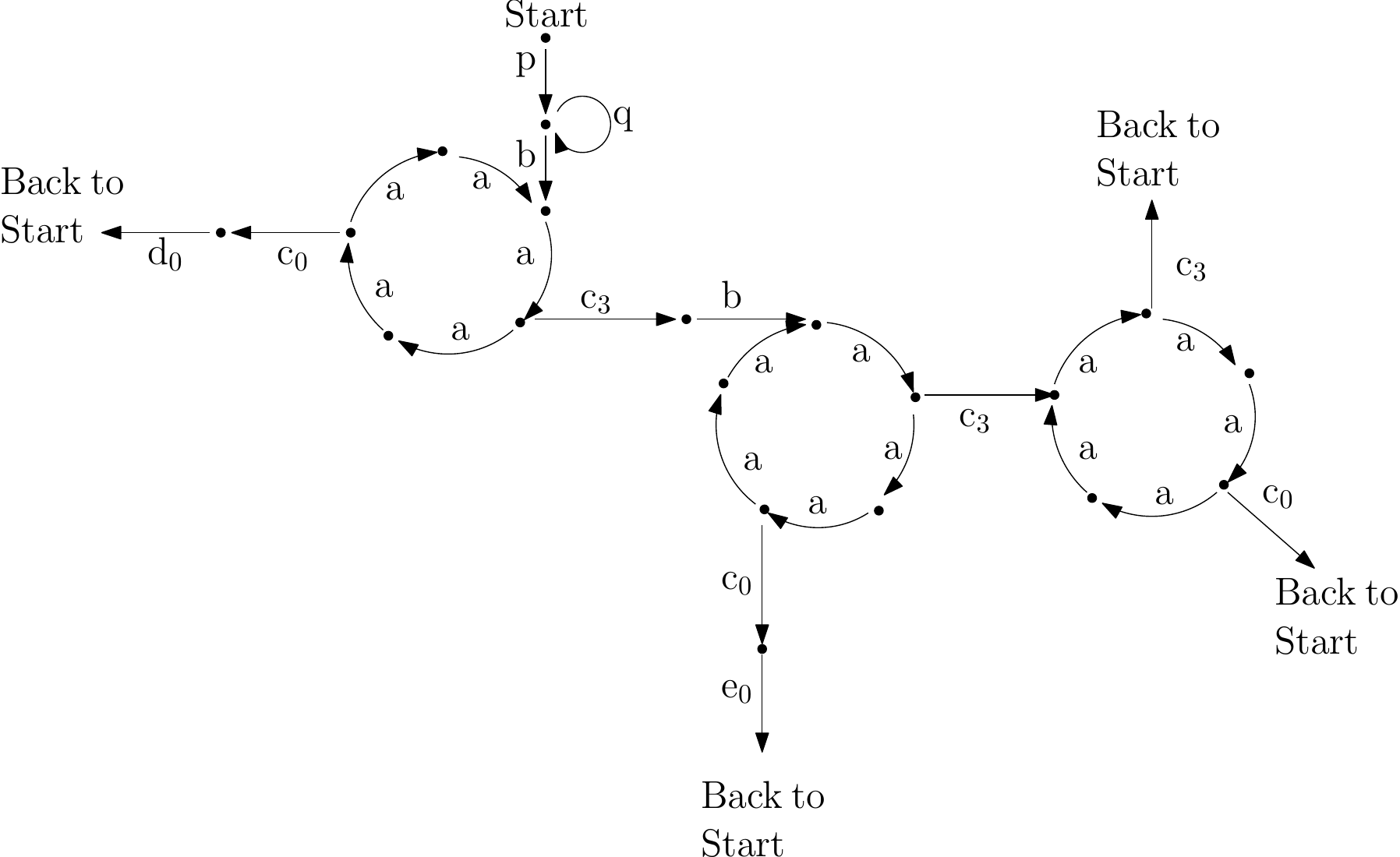}
\caption{$\mathcal{G}_{5, \{0,3\}}$ with $i^* = 3$ (where we use the convention that edges terminating at ``Back to Start" return to the initial vertex of the edge $p$.)}\label{G503}
\end{figure}

We first make some basic observations about the graphs $\mathcal{G}_{n,S}$. It is easy to check that for any $n, S$, the graph $\mathcal{G}_{n,S}$ will be irreducible, right-resolving and left-resolving, follower-separated and extender-separated. We furthermore observe that the graph is primitive:

\begin{lemma}\label{primitivity}
For any $n \in \mathbb{N}, S \subset \{0, 1, ..., n-1\}$, the graph $\mathcal{G}_{n,S}$ is primitive, and the primitivity distance of $\mathcal{G}_{n,S}$ is at most $3n + 3$.
\end{lemma}

\begin{proof}
Any irreducible graph with a self-loop is primitive, so due to the self-loop labeled $q$, $\mathcal{G}_{n,S}$ is primitive. So long as a path passes the vertex at which $q$ is anchored, that path may be inflated to any greater length by following the self-loop $q$ repeatedly. Given any two vertices $I$ and $J$ in $\mathcal{G}_{n,S}$, we may clearly get from $I$ to $J$, while being certain to also pass the vertex anchoring $q$, by traveling through each loop of edges labeled $a$ at most once (and following at most $n-1$ of the edges in each loop), and using no more than six letters total for the connecting paths between the loops. Thus, the longest path required to travel from one vertex to another in $\mathcal{G}_{n,S}$, requiring that such a path pass through the vertex which anchors $q$, is of length at most $3(n-1) + 6 = 3n+3$. For instance, if $S = \{1\}$, the shortest possible path from the initial vertex of $p$ to itself is labeled by the word $pba^{n-1}c_1ba^{n-1}c_1a^{n-1}c_1$, which clearly travels through the vertex anchoring $q$. This is a sort of ``worst-case" scenario, where, because $|S| = 1$, the shortest path requires traveling through all three loops of the graph, and with $i^*=1$, the exit of each loop is as far away from the entry point as possible.
\end{proof}

Any graph $\mathcal{G}_{n,S}$ created by the above construction will yield a shift with a follower set sequence and extender set sequence of period $n$ with a difference in $\limsup$ and $\liminf$ of 1:

\begin{theorem}\label{Lemma1} 
Let $n \in \mathbb{N}$ and $S \subset \{0, 1, ... , n-1\}$. Then the shift $X_{\mathcal{G}_{n, S}}$ has $3n + 3|S| + 4$ follower sets for words of length $\ell$ where $\ell \geq n + 2$ and $\ell \pmod n \in S$, and only $3n + 3|S| + 3$ follower sets for words of length $\ell'$ where $\ell' \geq n + 2$ and $\ell' \pmod n \notin S$. Furthermore, the shift $X_{\mathcal{G}_{n, S}}$ has $(3n + 2|S| + 1)^2 + |S| + 3$ extender sets for words of length $\ell$ where $\ell \geq 3n+3$ and $\ell \pmod n \in S$, and only $(3n + 2|S| + 1)^2 + |S| + 2$ extender sets for words of length $\ell'$ where $\ell' \geq 3n+3$ and $\ell' \pmod n \notin S$.
\end{theorem}

\begin{proof}
Let $\ell \geq n+2$. Let $\mathcal{G} = \mathcal{G}_{n, S}$ as defined above. We will find the number of follower sets and extender sets of words of length $\ell$ in $X_\mathcal{G}$ (though for one part of the extender set case, we will need to require $\ell \geq 3n+3$). Each follower set $F_X(w)$ is uniquely determined by the set of terminal vertices of paths labeled $w$ in $\mathcal{G}$, and each extender set $E_X(w)$ is determined by the set of pairs $\{I \rightarrow T \}$ of initial and terminal vertices of paths labeled $w$ in $\mathcal{G}$. \newline
\indent Since in $\mathcal{G}$ the words $p$, $q$, $c_{i^*}b$, $d_i$, $e_i$, and $c_{i^*}a$ are right-synchronizing, the longest path required to get from a right-synchronizing word to any vertex of $\mathcal{G}$ is $n + 2$. Because $\ell \geq n+2$, and the graph $\mathcal{G}$ is irreducible and right-resolving, every singleton represents the follower set of some word $w$ of length $\ell$. There are $3n + 2|S| + 1$ such follower sets, all distinct as $\mathcal{G}$ is follower-separated. \newline
\indent It is easy to see that the right-synchronizing words listed above are in fact bi-synchronizing. Since the graph $\mathcal{G}$ is both left- and right-resolving, if a legal word $w$ contains any bi-synchronizing word, then only one pair $\{ I \rightarrow T \}$ can be the initial and terminal vertices of paths labeled $w$. If $\ell \geq 3n+3$, by Lemma ~\ref{primitivity} every single pair $\{I \rightarrow T \}$ corresponds to the extender set of some word $w$ of length $\ell$. There are $(3n + 2|S| + 1)^2$ such extender sets, all distinct as $\mathcal{G}$ is extender-separated.
\indent Note that for the graph $\mathcal{G}$, recording the labels of two edges beyond any loop of edges labeled $a$, whether before or after the loop, results in a bi-synchronizing word. Since any word which would be capable of having a follower set not corresponding to a singleton (or of having an extender set not corresponding to a single pair of initial and terminal vertices) must be one which avoids all bi-synchronizing words, and $\ell \geq n+2>2$, any word of length $\ell$ with such a follower or extender set must include a string of $a$'s, and no more than 1 letter on either side of such a string. So only words of the forms $a^\ell$, $ka^{\ell-1}$, $a^{\ell-1}k'$, and $ka^{\ell-2}k'$ (where $k$ and $k'$ are labels appearing in $\mathcal{G}$ not equal to $a$) can terminate (or begin) at more than one vertex. \newline
\indent The word $a^\ell$ has 1 follower set, corresponding to all $3n$ vertices involved in loops of edges labeled $a$. This follower set is distinct from those corresponding to singletons, for which we have previously accounted. Similarly, the extender set of the word $a^\ell$ is distinct from those for which we have previously accounted.\newline
\indent The label $a$ is only followed in $\mathcal{G}$ by the labels $a$ and $c_i$ for all $i \in S$. For each $c_i$, the word $a^{\ell-1}c_i$ has a unique follower set corresponding to three terminal vertices, one for each loop in which the $a^{\ell-1}$ may occur. Thus there are $|S|$-many follower sets of this form, all distinct from previous follower sets. Similarly, there are $|S|$-many extender sets of this form, all distinct from previous extender sets as well.\newline
\indent The label $a$ is only preceded in $\mathcal{G}$ by the labels $a$, $b$, and $c_{i^*}$. Since $c_{i^*}a$ is bi-synchronizing, the follower set for the word $c_{i^*}a^{\ell-1}$ corresponds to a singleton and has already been counted. The word $ba^{\ell-1}$ has a follower set corresponding to two terminal vertices, one in each loop of edges labeled $a$ which is preceded by $b$. Hence there is 1 additional distinct follower set of this form, and again this behavior is mirrored by the extender sets--there is 1 additional distinct extender set for the word $ba^{\ell -1}$. \newline
\indent Finally, based on our above observations, if a word of the form $ka^{\ell-2}k'$ is to have a follower set corresponding to a greater number of vertices than one, that word must be of the form $ba^{\ell-2}c_i$. By construction, a path with this label only exists in $\mathcal{G}$ if $\ell \pmod n \in S$. If such a path exists, it contributes a single new follower set corresponding to two terminal vertices, each one edge past a loop of edges labeled $a$ that is preceded by the label $b$. This follower set cannot repeat one that we already found: if $i \neq i^*$, then the follower set of $ba^{\ell -2}c_i$ is exactly the set of all legal sequences beginning with $d_i$ or $e_i$, clearly not equal to the follower set of any other word of length $\ell$. If $i = i^*$, the follower set of $ba^{\ell -2}c_i$ contains sequences beginning with each of the letters $a$ and $b$, but no other letter, again setting it apart from any other follower set previously discussed. Similarly, the word $ba^{\ell -2}c_i$ contributes a single new extender set for any length $\ell$ for which a path with this label exists. \newline
\indent Therefore, in $X_\mathcal{G}$, if $\ell \geq n+2$ and $\ell \pmod n \in S$, there are $3n + 2|S| + 1 + 1 + |S| + 1 + 1 = 3n + 3|S| + 4$ follower sets of words of length $\ell$, while if $\ell' \geq n+2$ and $\ell' \pmod n \notin S$, there are only $3n + 3|S| + 3$ follower sets of words of length $\ell'$. Moreover, if $\ell \geq 3n+3$ and $\ell \pmod n \in S$, there are $(3n + 2|S| + 1)^2 + 1 + |S| + 1 + 1 = (3n + 2|S| + 1)^2 + |S| + 3$ extender sets of words of length $\ell$, while if $\ell' \geq 3n+3$ and $\ell' \pmod n \notin S$, there are only $(3n + 2|S| + 1)^2 + |S| + 2$ extender sets of words of length $\ell'$. \newline
\end{proof}

So, for the shift presented by the graph in Figure~\ref{G503}, for $\ell \geq 7$, the follower set sequence oscillates between $|F_X(\ell)| = 3(5) + 3(2) + 4 = 25$ if $\ell \equiv 0$ or $3 \pmod 5$, and $|F_X(\ell)| = 24$ if $\ell \equiv 1$, $2$, or $4 \pmod 5$. Moreover, for $\ell \geq 18$, the extender set sequence oscillates between $|E_X(\ell)| = (3(5) + 2(2) + 1)^2 + (2) + 3 = 405$ if $\ell \equiv 0$ or $3 \pmod 5$, and $|E_X(\ell)| = 404$ if $\ell \equiv 1$, $2$, or $4 \pmod 5$. \newline

We have now demonstrated the existence of sofic shifts whose follower set sequences and extender set sequences eventually oscillate between two different (but adjacent) values. We may furthermore combine these graphs, forming new graphs presenting shifts whose follower set sequences and extender set sequences oscillate by more than 1:

\begin{theorem}\label{Lemma2}
Let $\mathcal{G}_1$ and $\mathcal{G}_2$ be two finite, irreducible, right-resolving, left-resolving, primitive, extender-separated labeled graphs with disjoint label sets, each containing a self-loop labeled by a bi-synchronizing letter $q_1$ and $q_2$ respectively. Let $I_1$ be the anchoring vertex of $q_1$ in $\mathcal{G}_1$ and $I_2$ be the anchoring vertex of $q_2$ in $\mathcal{G}_2$. Let $x,y$ be letters not in the label set of $\mathcal{G}_1$ or $\mathcal{G}_2$. Construct a new graph $\mathcal{G}$ by taking the disjoint union of $\mathcal{G}_1$ and $\mathcal{G}_2$ and adding an edge labeled $x$ beginning at $I_1$ and terminating at $I_2$ and an edge labeled $y$ beginning at $I_2$ and terminating at $I_1$. Then $\mathcal{G}$ is finite, irreducible, right-resolving, left-resolving, primitive, extender-separated, contains a self-loop labeled with a bi-synchronizing letter, and for any $\ell \in \mathbb{N}$, $$|F_{X_\mathcal{G}}(\ell)| =  |F_{X_{\mathcal{G}_1}}(\ell)| + |F_{X_{\mathcal{G}_2}}(\ell)|.$$ Moreover, for any $\ell$ greater than twice the maximum of the  primitivity distances of $\mathcal{G}_1$ and $\mathcal{G}_2$, $$|E_{X_\mathcal{G}}(\ell)| = |E_{X_{\mathcal{G}_1}}(\ell)| + |E_{X_{\mathcal{G}_2}}(\ell)| + 2|V(\mathcal{G}_1)|\cdot |V(\mathcal{G}_2)|.$$ \newline
\end{theorem}

\begin{figure}[h]
\includegraphics[scale=.8]{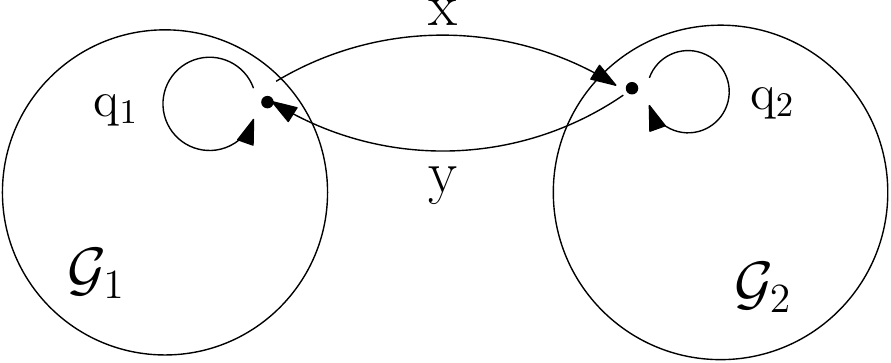}
\caption{The graph $\mathcal{G}$ constructed as in Theorem ~\ref{Lemma2}}
\end{figure}

\begin{proof}

The reader may check that $\mathcal{G}$ is finite, irreducible, right-resolving, left-resolving, primitive, extender-separated and contains a self-loop labeled with a bi-synchronizing letter. We first check that this construction does not cause any collapsing of follower or extender sets. That is, if two words $w$ and $v$ had distinct follower or extender sets in $X_{\mathcal{G}_1}\sqcup X_{\mathcal{G}_2}$, then they have distinct follower or extender sets in $X_{\mathcal{G}}$. We present the argument for follower sets; the argument for extender sets is similar. \newline
\indent If for two words $w$ and $v$ in $L_\ell(X_{\mathcal{G}_1})\sqcup L_\ell(X_{\mathcal{G}_2})$, we have $F(w) \neq F(v)$, then there exists a sequence $s$ in $X_{\mathcal{G}_1}$ or $X_{\mathcal{G}_2}$ which may follow one word but not the other--without loss of generality, say that $s$ may follow $w$ but not $v$. Such a sequence may still follow $w$ in $\mathcal{G}$ by following the same path labeled $s$ which followed $w$ in $\mathcal{G}_1$ or $\mathcal{G}_2$ to begin with. The sequence $s$ still may not follow $v$, as no path existed in the parent graph $\mathcal{G}_1$ or $\mathcal{G}_2$ labeled $s$ following $v$, and any new paths introduced by our construction may not be labeled $s$, as $s$ did not contain the letters $x$ or $y$. Hence $|F_{X_{\mathcal{G}_1}}(\ell)| + |F_{X_{\mathcal{G}_2}}(\ell)| \leq |F_{X_{\mathcal{G}}}(\ell)|$, and similarly, $|E_{X_{\mathcal{G}_1}}(\ell)| + |E_{X_{\mathcal{G}_2}}(\ell)| \leq |E_{X_{\mathcal{G}}}(\ell)|$. \newline
\indent Now we establish that no extra follower sets are introduced by this construction. If two words $w$ and $v$ had the same follower set in $X_{\mathcal{G}_1}\sqcup X_{\mathcal{G}_2}$, then they certainly exist in the same parent graph, $\mathcal{G}_1$ or $\mathcal{G}_2$; without loss of generality, say $\mathcal{G}_1$. Let $s$ be some sequence following $w$ in $\mathcal{G}$. If the path in $\mathcal{G}$ labeled $s$ is contained within $\mathcal{G}_1$, then $s$ is part of the follower set of $w$ in $X_{\mathcal{G}_1}$ and thus, is part of the follower set of $v$ in $X_{\mathcal{G}_1}$. So $s$ may follow $v$ in $\mathcal{G}$. On the other hand, if $s$ is presented by a path traveling through both graphs, then $s$ contains the letter $x$. Let $z$ denote the maximal finite prefix of $s$ without the letter $x$. A path labeled $z$ follows $w$ in $\mathcal{G}_1$, and since $z$ must terminate at $I_1$ in order to be followed by $x$, a path labeled $zq_1$ follows $w$ in $\mathcal{G}_1$ as well. Then a path labeled $zq_1$ must also follow $v$ in $\mathcal{G}_1$, and since $q_1$ is left-synchronizing, there must exist a path labeled $z$ following $v$ in $\mathcal{G}_1$ terminating at $I_1$. Such a path may certainly, then, be followed by $x$, and indeed, the remaining portion of $s$, so $s$ is in the follower set of $v$. Thus in $X_\mathcal{G}$, the follower sets of $w$ and $v$ remain the same. \newline
\indent Moreover, if a word $v$ is not in either $L_\ell(X_{\mathcal{G}_1})$ or $L_\ell(X_{\mathcal{G}_2})$, then $v$ includes either the letter $x$ or $y$. Since $x$ and $y$ are right-synchronizing, a path labeled $v$ may terminate only at a single vertex. Due to the fact that $\mathcal{G}_1$ and $\mathcal{G}_2$ are each irreducible, right-resolving, and contain a right-synchronizing letter, there exists a word $w$ in either $L_\ell(\mathcal{G}_1)$ or $L_\ell(\mathcal{G}_2)$ which terminates at the same unique vertex as paths labeled $v$. Since the right-synchronizing letter terminates at the same vertex as $x$ or $y$, depending on the graph, we may construct $w$ to be of the desired length $\ell$ in the following way: Let $v = v_1v_2...v_\ell$ and let $v_i$ be the last occurrence of $x$ or $y$ in $v$. Then set $w_{i+1}w_{i+2}...w_\ell = v_{i+1}v_{i+2}...v_\ell$.  Create $w_1...w_i$ by replacing $v_i$ by $q_1$ or $q_2$ (choosing the one which terminates at the same place as $v_i$) and then following any path backward in that same parent graph ($\mathcal{G}_1$ or $\mathcal{G}_2$) to fill in $i-1$ labels before $w_i$. Then $w$ is a right-synchronizing word contained entirely in either $\mathcal{G}_1$ or $\mathcal{G}_2$ of length $\ell$ terminating at the same single vertex as $v$, and so $w$ and $v$ have the same follower set in $X_\mathcal{G}$. Therefore the construction introduced no extra follower sets, and $|F_{X_\mathcal{G}}(\ell)| =  |F_{X_{\mathcal{G}_1}}(\ell)| + |F_{X_{\mathcal{G}_2}}(\ell)|.$ \newline
\indent This construction also causes no splitting of extender sets: If $w$ and $v$ have the same extender set in $X_{\mathcal{G}_1}\sqcup X_{\mathcal{G}_2}$, then they certainly exist in the same parent graph, $\mathcal{G}_1$ or $\mathcal{G}_2$; without loss of generality, say $\mathcal{G}_1$. Let $(s,u)$ be in the extender set of $w$. If $s$ and $u$ are both contained within $\mathcal{G}_1$, then $(s,u)$ is certainly in the extender set of $v$ as well. Otherwise, let $z$ be the maximal suffix of $s$ with no appearance of the letter $y$ and $z'$ be the maximal prefix of $u$ with no appearance of $x$. (Note that in this case one of $z$ and $z'$ may be infinite, but not both.) Then $zwz'$ is contained in $\mathcal{G}_1$ and since $w$ and $v$ have the same extender set in $\mathcal{G}_1$, a path labeled $zvz'$ exists in $\mathcal{G}_1$ as well. If $z = s$ but $z'$ is finite, then as above, $z'$ may be followed by $q_1$ in $\mathcal{G}_1$, which is left-synchronizing, so there must exist a path labeled $z'$ following $sv$ in $\mathcal{G}_1$ terminating at $I_1$. Such a path may certainly, then, be followed by $x$, and the remaining portion of $u$, and so $(s,u)$ is in the extender set of $v$. Similarly, if $z$ is finite but $z' = u$, then $z$ may be preceded by $q_1$, which is right-synchronizing, so there must exist a path labeled $z$ preceding $vu$ in $\mathcal{G}_1$ beginning at $I_1$. Such a path may then be preceded by $y$, and the preceding portion of $s$, and so $(s,u)$ is in the extender set of $v$. If both $z$ and $z'$ are finite, then the path $q_1zvz'q_1$ exists in $\mathcal{G}_1$, and so a path labeled $zvz'$ exists in $\mathcal{G}_1$ beginning and ending at $I_1$ which then may be extended to an infinite path labeled $svu$, so $(s,u)$ is in the extender set of $v$. Thus $E(w) = E(v)$.  \newline
\indent Finally, if a word $v$ is not in either $L_\ell(X_{\mathcal{G}_1})$ or $L_\ell(X_{\mathcal{G}_2})$, then $v$ includes either the letter $x$ or $y$. Since $x$ and $y$ are bi-synchronizing and $\mathcal{G}$ is both left- and right-resolving, paths labeled $v$ have exactly one pair $\{ I \rightarrow T \}$ of initial and terminal vertices. If $I$ and $T$ are in the same parent graph $\mathcal{G}_i$, we observe that if $\ell$ is longer than twice the primitivity distance for $\mathcal{G}_i$, we can construct a path $w$ from $I$ to $I_i$, and a path $u$ from $I_i$ to $T$, such that the path labeled $wq_iu$ has length $\ell$. Because $q_i$ is bi-synchronizing, paths labeled $wq_iu$ have only one pair of initial and terminal vertices; the exact same initial and terminal vertices as $v$, so $E(v) = E(wq_iu)$. On the other hand, if $I$ and $T$ are in different parent graphs, then $E(v)$ is certainly not equal to the extender set of any word in $L_\ell(X_{\mathcal{G}_1})\sqcup L_\ell(X_{\mathcal{G}_2})$, so the construction did introduce new extender sets of words of length $\ell$, but only at most $2|V(\mathcal{G}_1)|\cdot |V(\mathcal{G}_2)|$ many of them. Furthermore, if $\ell$ is longer than the primitivity distance of $\mathcal{G}$, then all $2|V(\mathcal{G}_1)|\cdot |V(\mathcal{G}_2)|$ such extender sets will be realized, and since $\mathcal{G}$ is extender-separated, they will all be distinct. The primitivity distance for $\mathcal{G}$ is at most one greater than the sum of the primitivity distances of $\mathcal{G}_1$ and $\mathcal{G}_2$, so if $\ell$ is greater than twice the maximum of the primitivity distances of $\mathcal{G}_1$ and $\mathcal{G}_2$, we have $|E_{X_\mathcal{G}}(\ell)| = |E_{X_{\mathcal{G}_1}}(\ell)| + |E_{X_{\mathcal{G}_2}}(\ell)| + 2|V(\mathcal{G}_1)|\cdot |V(\mathcal{G}_2)|.$
\end{proof}

It is evident that the process outlined in Theorem~\ref{Lemma2} may be repeated an arbitrary number of times, and since the constant introduced ($0$ in the follower set case, $2|V(\mathcal{G}_1)|\cdot |V(\mathcal{G}_2)|$ in the extender set case) does not depend on $\ell$, we may use this process to increase the oscillations in the follower and extender set sequences of the resulting shift. We formalize this idea in Theorems~\ref{mainthm} and ~\ref{ESmain}, which state that there exist sofic shifts with follower set sequences and extender set sequences of every eventual period and with any natural number as the difference in $\limsup$ and $\liminf$ of the sequence.

\begin{proof}[Proof of Theorem~\ref{mainthm}]
Let $\mathcal{G}_{n,S}$ denote the graph constructed from $n$ and $S \subset \{0,1,..., n-1\}$ as in Theorem~\ref{Lemma1}. First construct $\mathcal{G}_{n,A_2\cup A_3\cup ...\cup A_k}$. By Theorem~\ref{Lemma1}, this graph will give one more follower set to words of length $\ell \geq n+2$ and $\ell \pmod n \in A_2\cup A_3\cup ...\cup A_k$ than to words of length $\ell \geq n+2$ and $\ell \pmod n \in A_1$. \newline
\indent Now, as $\mathcal{G}_{n, A_2\cup A_3 \cup ... \cup A_k}$ is finite, irreducible, right-resolving, left-resolving, primitive, extender-separated, and contains a self-loop labeled with the bi-synchronizing letter $q$, we may use the process defined in Theorem~\ref{Lemma2} to join together $r_2$ many copies of $\mathcal{G}_{n, A_2\cup A_3 \cup ... \cup A_k}$ only by giving each copy a disjoint set of labels. Call the resulting graph $\mathcal{G}_2$. For each $\ell$, the number of follower sets of words of length $\ell$ in $\mathcal{G}_2$ is the sum of the number of follower sets of words of length $\ell$ in each of the $r_2$ copies of $\mathcal{G}_{n, A_2\cup A_3 \cup ... \cup A_k}$. Therefore, the graph $\mathcal{G}_2$ gives $r_2$ more follower sets to words of length $\ell \geq n+2$ with $\ell \pmod n \in A_2\cup A_3\cup ... \cup A_k$ than to words of length $\ell \geq n+2$ with $\ell \pmod n \in A_1$. \newline
\indent Using the same process, we may now join onto $\mathcal{G}_2$ another $(r_3-r_2)$ many copies of the graph $\mathcal{G}_{n, A_3\cup A_4 \cup ... \cup A_k}$, and call the resulting graph $\mathcal{G}_3$. Now,  words of length $\ell \geq n+2$ where $\ell \pmod n \in A_3 \cup A_4 \cup ... \cup A_k$ will have $(r_3 - r_2) + r_2 = r_3$ more follower sets than words of length $\ell \geq n+2$ where $\ell \pmod n \in A_1$, while words of length $\ell \geq n+2$ where $\ell \pmod n \in A_2$ will have only $r_2$ greater follower sets than words of length $\ell \geq n+2$ where $\ell \pmod n \in A_1$. \newline
\indent Continue on this way, adjoining next $(r_4 - r_3)$ copies of $\mathcal{G}_{n, A_4 \cup A_5\cup ... \cup A_k}$ to make $\mathcal{G}_4$, and so forth, terminating after constructing $\mathcal{G}_k$. The graph will clearly be irreducible, and in $\mathcal{G}_k$, for each $1 \leq j \leq k$, words of length $\ell \geq n+2$ where $\ell \pmod n \in A_j$ will have $r_j$ more follower sets than words of length $\ell \geq n+2$ where $\ell \pmod n \in A_1$. That is, if $m$ is defined to be the number of follower sets of words of length $\ell \geq n+2$ where $\ell \pmod n \in A_1$, then words of length $\ell \geq n+2$ where $\ell \pmod n \in A_j$ will have $m + r_j$ many follower sets in $X_{\mathcal{G}_k}$. \newline
\indent Using the formula established in Theorem~\ref{Lemma1}, we see that for each $\mathcal{G}_{n,S}$, words of lengths $\ell \geq n+2$ whose residue classes are not in $S$ have $3n + 3|S| + 3$ many follower sets. Since $A_1 \subseteq S^c$ for every graph used in the construction of $\mathcal{G}_k$, words of length $\ell \geq n+2$ where $\ell \pmod n \in A_1$ must have the following number of follower sets: \newline
\begin{align*}
m &= r_2(3n + 3|A_2\cup A_3\cup ... \cup A_k| + 3) + (r_3 - r_2)(3n + 3|A_3\cup A_4\cup ... \cup A_k| + 3) + ... \\
&\phantom{====================}+ (r_k - r_{k-1})(3n + 3|A_k| + 3) \\
&= r_2(3n + 3\sum_{j=2}^k|A_j| + 3) + (r_3 - r_2)(3n + 3\sum_{j=3}^k|A_j| + 3) + ...\\
&\phantom{====================}+ (r_k - r_{k-1})(3n+ 3\sum_{j = k}^k|A_j| + 3) \\
&= \sum_{i=2}^k(r_i - r_{i-1})(3n + 3\sum_{j = i}^k|A_j| + 3).
\end{align*}
Furthermore, since for all $i \geq 2$, we have $\displaystyle \sum_{j = i}^k|A_j| < n$, we get that \newline
\begin{align*}
m &< \sum_{i=2}^k(r_i - r_{i-1})(6n + 3) \\
&=(6n + 3) \sum_{i=2}^k(r_i - r_{i-1}) \\
&= (6n+3)r_k.
\end{align*}
\end{proof}

This theorem shows that we may construct a sofic shift whose follower set sequence follows any desired oscillation scheme--increasing or decreasing by specified amounts at specified lengths $\ell$, and repeating with any desired eventual period. A similar result holds for extender set sequences, though the bounds for $m = \liminf \{|E_X(\ell)|\}$ and for the start of the periodicity of the sequence are different. 

\begin{proof}[Proof of Theorem~\ref{ESmain}]
We follow the same construction as in the proof of Theorem~\ref{mainthm}, combining $r_k$-many graphs using Theorem~\ref{Lemma2} to construct the graph $\mathcal{G}_k$. Then, in $\mathcal{G}_k$, for each $1 \leq j \leq k$, sufficiently long words of length $\ell$ where $\ell \pmod n \in A_j$ will have $r_j$ more extender sets than sufficiently long words of length $\ell$ where $\ell \pmod n \in A_1$. That is, if $m$ is defined to be the number of extender sets of words of sufficiently long words of length $\ell$ where $\ell \pmod n \in A_1$, then words of sufficient length $\ell$ where $\ell \pmod n \in A_j$ will have $m + r_j$ many extender sets in $X_{\mathcal{G}_k}$ for all $1 \leq j \leq k$.  \newline
\indent To discover what length is sufficient for periodicity of the extender set sequence to begin, we observe that for every graph $\mathcal{G}$ used in the construction of $\mathcal{G}_k$, the primitivity distance of $\mathcal{G}$ is less than or equal to $3n + 3$ as in Lemma ~\ref{primitivity}. Note that, since $n \geq 1$, we have $3n + 3 \leq 7n -1$. (In fact, in any interesting case, $n \geq 2$, so $3n + 3 \leq 5n -1$, but $n = 1$ certainly may be chosen as a trivial case, where $S = \{0\}$ necessarily). By Theorem ~\ref{Lemma2}, the eventual periodicity of the extender set sequence of the combination of two graphs begins before the $2z + 1^{st}$ term, where $z$ is the maximum of the primitivity distances of the two graphs. So, when adding two graphs together in this construction, we get that the eventual periodicity begins before $(3n + 3) + (3n + 3) + 1 \leq (7n -1) + (7n -1) + 1 = 14n -1$. (We observe that the primitivity distance of the resulting graph will also be less than $14n -1$). Since $n$ is the same for each graph involved in the construction, it does not matter which type of graph we are adding at each step, whether a copy of $\mathcal{G}_{n, A_2\cup...\cup A_k}, \mathcal{G}_{n, A_3\cup ...\cup A_k}$, up to $\mathcal{G}_{n, A_k}$. \newline
\indent Though we have performed the same construction as in Theorem ~\ref{mainthm}, we add our graphs in a more efficient order to minimize the effect of the constant $2|V(\mathcal{G}_1)|\cdot |V(\mathcal{G}_2)|$. First consider the case where $r_k$ is a power of 2. Then we may choose to construct $\mathcal{G}_k$ in such a way that at each step, we add two graphs each made up of the same number of components. (2 graphs each consisting of 2 components to make 4, 2 graphs each consisting of 4 components to make 8, and so on). Then if $a_i$ is an upper bound for the start of primitivity at the $i^{th}$ step, an upper bound for the primitivity at the $i+1^{st}$ step is $2(a_i)+1$. Then letting $a_1 = 7n -1$, the value of the sequence $a_i = 2(a_{i-1}) + 1$ at $i = \log_2(r_k) + 1$ will give an upper bound for the start of primitivity for $\mathcal{G}_k$, since we must add together two peices of equal components $\log_2(r_k)$ times to construct $\mathcal{G}_k$. \newline
\indent We claim that for all $i$, $a_i = 2^{i-1}(7n) -1$. This is trivially true for the base case, $i = 1$. By induction, suppose $a_{i-1} = 2^{i-2}(7n) - 1$. Then 
\begin{align*}
a_i &= 2(a_{i-1}) +1 \\
&= 2(2^{i-2}(7n) -1) + 1 \\
&= 2^{i-1}(7n) - 2 + 1 \\
&= 2^{i-1}(7n) -1.
\end{align*}
Thus, when $r_k$ is a power of 2, the primitivity of the extender set sequence of $\mathcal{G}_k$ begins before $a_{\log_2(r_k) + 1} = 2^{\log_2(r_k)}(7n) - 1 = 7nr_k -1$. \newline
\indent Now, the upper bound for the beginning of the periodicity of the extender set sequence certainly increases as $r_k$ increases--increasing $r_k$ means adding more graphs to construct $\mathcal{G}_k$--and so, since $r_k \leq 2^{\lceil \log_2(r_k)\rceil}$ for all $r_k \in \mathbb{N}$, and $2^{\lceil \log_2(r_k)\rceil}$ is a power of 2, for any $r_k$, the primitivity of $\mathcal{G}_k$ must begin at the latest when $\ell = 7(2^{\lceil \log_2(r_k)\rceil})n -1 \leq 7(2^{\log_2(r_k) + 1})n - 1 = 14(r_k)n -1$. \newline
\indent Finally, it remains to show that in this construction, $m \leq 39n^2r_k^2$. We first show that $\mathcal{G}_2$ (that is, $r_2$ combined copies of $\mathcal{G}_{n, A_2 \cup...\cup A_k}$) has at most $6nr_2$ vertices and will have $m \leq 39n^2r_2^2$. The bound on the number of vertices is clear--for any graph $\mathcal{G}_{n, S}$ constructed by the method defined at the beginning of this section, $\mathcal{G}_{n,S}$ has $3n + 2|S| + 1$ vertices, and since $|S| \leq n$ and $n \geq 1$, we have $3n + 2|S| + 1 \leq 6n$. With each graph having at most $6n$ vertices, it is trivial that $\mathcal{G}_2$ has at most $6nr_2$ vertices. As discussed in Theorem ~\ref{Lemma1}, the number of extender sets for $\ell \geq 3n+3$ and $\ell \pmod n \notin S$ (that is, $m$ for $\mathcal{G}_{n,S}$) is $(3n + 2|S| + 1)^2 + |S| + 2 \leq 36n^2 + n + 2 \leq 39n^2$. This proves the base case, when $r_2 = 1$. Suppose for an induction that after joining together $i$ copies of $\mathcal{G}_{n, A_2 \cup...\cup A_k}$ to make a graph $\mathcal{G}$, we get $m \leq 39i^2n^2$, and we then adjoin a single copy of $\mathcal{G}_{n, A_2\cup...\cup A_k}$ to $\mathcal{G}$. Then, by Theorem ~\ref{Lemma2}, words of sufficient length $\ell$ in the new graph where $\ell \pmod n \in A_1$ will have a number of extender sets equal to the number of extender sets for words of such length in $\mathcal{G}$ (bounded above by $39i^2n^2$) plus the number of extender sets for words of such length in $\mathcal{G}_{n, A_2\cup ... \cup A_k}$ (bounded above by $39n^2$) plus twice the product of the number of vertices in $\mathcal{G}$ and $\mathcal{G}_{n, A_2\cup ... \cup A_k}$ (bounded above by $2(6in)(6n) < 2i(39n^2)$). \newline
\indent Thus, for the resulting graph containing $i + 1$ copies of $\mathcal{G}_{n, A_2\cup ... \cup A_k}$, we have:
$$m < 39i^2n^2 + 39n^2 + 2i(39n^2) = 39n^2(i^2 + 1 + 2i) = 39n^2(i+1)^2,$$
giving the result for $\mathcal{G}_2$ when $i = r_2$. \newline
\indent We next consider adding $(r_3-r_2)$ many copies of $\mathcal{G}_{n, A_3\cup ...\cup A_k}$ to $\mathcal{G}_2$ to make $\mathcal{G}_3$. By the same argument as above, the graph consisting of $(r_3-r_2)$ many copies of $\mathcal{G}_{n, A_3\cup ...\cup A_k}$ will have at most $6n(r_3-r_2)$ vertices and $m \leq 39n^2(r_3-r_2)^2$. Again using Theorem ~\ref{Lemma2} to combine the two graphs, the resulting graph $\mathcal{G}_3$ will have at most $6nr_2 + 6n(r_3 - r_2) = 6nr_3$ vertices, and will have
\begin{align*}
m &\leq 39n^2r_2^2 + 39n^2(r_3-r_2)^2 + 2(6nr_2)(6n(r_3-r_2))\\
 &< 39n^2r_2^2 + 39n^2(r_3-r_2)^2 + 39n^2(2(r_2)(r_3-r_2)) \\
 &= 39n^2(r_2^2 + (r_3-r_2)^2 + 2((r_2)(r_3-r_2)) \\
 &= 39n^2(r_2 + (r_3-r_2))^2 \\
 &= 39n^2r_3^2.
\end{align*}
Continuing inductively, we can see that for $\mathcal{G}_k$, we will have $m \leq 39n^2r_k^2$.
\end{proof}

While we may achieve any desired oscillation scheme, we cannot achieve any eventually periodic sequence we like--$m$ must be sufficiently large.

\begin{theorem}\label{mlowerbds}
Let $X$ be an irreducible sofic shift with $\liminf\{|F_X(\ell)|\} = m$ and $ \limsup \{|F_X(\ell)|\} = m +r$ with least eventual period $n$.  Then we have that $m > \log_2(r)$ and $m > \frac{1}{2}\log_2(\log_2(n))$. If $\liminf \{|E_X(\ell)|\} = m'$ and $\limsup \{|E_X(\ell)|\} = m' + r'$ with least eventual period $n'$, then $m' > \sqrt{\log_2(r')}$ and $m' > \sqrt{\frac{1}{2}\log_2(\log_2(n'))}$.
\end{theorem}

\begin{proof}
Let $\mathcal{G}$ be an irreducible right-resolving presentation of $X$ which contains a right-synchronizing word. Let $|V(\mathcal{G})|$ be denoted by $V$. Since a follower set of a word $w$ in a sofic shift is determined by the non-empty set of terminal vertices of paths labeled $w$ in a presentation $\mathcal{G}$, $X$ has less than $|\mathcal{P}(V(\mathcal{G}))| = 2^V$ follower sets, and so $m + r < 2^V$. Because $\mathcal{G}$ is irreducible, right-resolving, and contains a right-synchronizing word, for large enough $\ell$, each singleton will correspond to the follower set of some word for any length greater than $\ell$. Because $m$ occurs infinitely often in the follower set sequence, $m$ must be greater than or equal to the number of singletons in $\mathcal{G}$, that is, $m \geq V$. Thus, we have:
$$2^m \geq 2^V > m + r > r$$
$$m > \log_2(r).$$
Moreover, as there are less than $2^V$ follower sets, the least eventual period $n$ of the follower set sequence is less than or equal to $(2^V)!$, as in Theorem ~\ref{eventuallyperiodic}. Thus $(2^V)! \geq n$. So we have:
$$(2^m)^{(2^m)} > (2^m)! \geq (2^V)! \geq n$$
$$(2^m)^{(2^m)} > n $$
$$\log_2((2^m)^{(2^m)}) > \log_2(n)$$
$$2^m\log_2(2^m) > \log_2(n)$$
$$2^m(m) > \log_2(n)$$
$$\log_2(2^m) + \log_2(m) > \log_2(\log_2(n))$$
$$m + \log_2(m) > \log_2(\log_2(n)).$$
Since $m > \log_2(m)$, we have:
$$2m > \log_2(\log_2(n))$$
$$m > \frac{1}{2}\log_2(\log_2(n)).$$
Now, an extender set of a word $w$ in a sofic shift is determined by the non-empty set of pairs of initial and terminal vertices of paths labeled $w$ in $\mathcal{G}$, so $X$ has less than $2^{V^2}$ extender sets, that is, $m' + r' < 2^{V^2}$. Since two words with the same extender set have the same follower set, $m' \geq m \geq V$. Thus we have:
$$2^{{m'}^2} \geq 2^{V^2} > m' + r' > r'$$
$${m'}^2 > \log_2(r')$$
$$m' > \sqrt{\log_2(r')}.$$
Finally, as there are less than $2^{V^2}$ extender sets, by Theorem~\ref{eventuallyperiodic}, $(2^{v^2})! \geq n'$, giving:
$$(2^{{m'}^2})^{(2^{{m'}^2})} > (2^{{m'}^2})! \geq (2^{V^2})! \geq n'$$
$$(2^{{m'}^2})^{(2^{{m'}^2})} > n'$$
$$(2^{{m'}^2})({m'}^2) > \log_2(n')$$
$${m'}^2 + \log_2({m'}^2) > \log_2(\log_2(n')).$$
Since ${m'}^2 > \log_2({m'}^2)$, we have:
$$2{m'}^2 > \log_2(\log_2(n'))$$
$$m' > \sqrt{\frac{1}{2}\log_2(\log_2(n'))}.$$
\end{proof}

\begin{remark}
For the examples discussed in this paper, we may take $r$ or $r' = r_k$.
\end{remark}

\section{The Non-Sofic Case}
\label{XcrossY}

In this section we demonstrate the existence of a non-sofic shift whose follower set sequence and extender set sequence are not monotone increasing. The construction uses the following fact about Sturmian shifts (for a definition, see ~\cite{Fogg}):

\begin{lemma}\label{Sturmian}
If $Y$ is a Sturmian shift, then for any $\ell \in \mathbb{N}$, $|F_Y(\ell)| = |E_Y(\ell)| = \ell + 1$.
\end{lemma}

\begin{proof}
From ~\cite{Fogg}, it is known that for a fixed length $\ell$, Sturmian shifts have exactly $\ell + 1$ words in $L_\ell(Y)$, so it is sufficient to show that any two words of length $\ell$ in $Y$ have distinct follower and extender sets. We also know from ~\cite{Fogg} that Sturmian shifts are symbolic codings of irrational circle rotations, say by $\alpha \notin \mathbb{Q}$, where the circle is expressed by the interval $[0,1]$ with the points $0$ and $1$ identified, and the interval coded with $1$ has length $\alpha$. We may take $\alpha < \frac{1}{2}$ by simply switching the labels $0$ and $1$ whenever $\alpha > \frac{1}{2}$. Furthermore, the cylinder sets of words of length $\ell$ correspond to a partition of the circle into $\ell + 1$ subintervals, so for two words of length $\ell$ in $Y$, each corresponds to a subinterval of the circle, and the two subintervals are disjoint. Let $w$ and $v$ be two distinct words in $L_\ell (Y)$ corresponding to disjoint intervals $I_w$ and $I_v$, $[0, \alpha)$ be the interval coded with $1$, and $T_\alpha$ be the rotation by $\alpha$. We claim that there exists an $N \in \mathbb{N}$ such that $T_\alpha^{-N}[0,\alpha)$ intersects one of $I_w$ and $I_v$ but not the other: Since $\alpha < \frac{1}{2}$, and since $\{n\alpha \ | \ n \in \mathbb{N} \}$ is dense in the circle, if one of $I_w$ and $I_v$ has length at least $\frac{1}{2}$, there exists $N \in \mathbb{N}$ such that $T_\alpha^{-N}[0, \alpha)$ is contained entirely inside that large interval, and thus completely disjoint from the other. Otherwise, take $I_w^c$, which clearly has length at least $\frac{1}{2}$, and find an $N \in \mathbb{N}$ such that $T_\alpha^{-N}[0,\alpha)$ is contained inside $I_w^c$ and intersects $I_v \subseteq I_w^c$, again possible due to denseness of $\{n\alpha \ | \ n \in \mathbb{N} \}$. Hence we have proved our claim, that $\exists \> N \in \mathbb{N}$ such that $T_\alpha^{-N}[0,\alpha)$ intersects one of $I_w$ and $I_v$ but not the other, and therefore, that the symbol $1$ may follow one of the words $w$ and $v$ exactly $N$ units later, but not the other. Therefore $w$ and $v$ have distinct follower sets, and thus, distinct extender sets, completing the proof.
\end{proof}

We use Lemma ~\ref{Sturmian} and the sofic shifts constructed in Section ~\ref{followers} to build a non-sofic shift with non-monotonically increasing follower and extender set sequences:

\begin{proof}[Proof of Theorem ~\ref{Non-sofic}]
Take a sofic shift $X = X_{\mathcal{G}_{n, S}}$ for any $n, S$ as defined in Section ~\ref{followers}. Take the direct product of $X$ and a Sturmian shift $Y$. Two words in $X \times Y$ have the same extender set if and only if the projection of those words to both their first and second coordinates have the same extender set in $X$ and $Y$, respectively. That is, if two words $w$ and $v$ have different extender sets in $X$, then any two words whose projections to their first coordinate are $w$ and $v$ will have different extender sets in $X \times Y$, and similarly for words $w'$ and $v'$ with different extender sets in $Y$. Therefore $|E_{X\times Y}(\ell)| = |E_X(\ell)|\cdot |E_Y(\ell)|$. By similar logic, $|F_{X\times Y}(\ell)| = |F_X(\ell)|\cdot |F_Y(\ell)|$. \newline
\indent Thus, if we let $m = \liminf \{|E_X(\ell)|\}_{\ell \in \mathbb{N}}$, then for any $\ell \geq 3n+3$ with $\ell \pmod n \notin S$, we have $|E_{X\times Y}(\ell)| = m\cdot (\ell + 1)$, and if $\ell \pmod n \in S$, then $|E_{X\times Y}(\ell)| = (m + 1)(\ell + 1)$. As $m$ is fixed and $\ell$ approaches infinity, it is clear that $\{|E_{X\times Y}(\ell)|\}$ is unbounded, and thus the shift $X \times Y$ is nonsofic. Furthermore, as the direct product of a mixing shift ($X$ is primitive by Lemma~\ref{primitivity}, and therefore mixing) with an irreducible shift, $X \times Y$ is irreducible. \newline
\indent Choose $\ell$ large enough that $\ell > m -1$, and such that $\ell \pmod n \in S$ and $\ell + 1 \pmod n \notin S$. Then 
\begin{align*}
|E_{X\times Y}(\ell)| &= (m + 1)(\ell + 1) \\
&= m\ell + m + \ell + 1 \\
&> m\ell + m + (m -1) +1 \\
&= m\ell +2m \\
&= m(\ell + 2) \\
&= |E_{X\times Y}(\ell +1)|.
\end{align*}
Therefore the extender set sequence of $X\times Y$ is not monotone increasing. A similar argument shows that the follower set sequence of $X\times Y$ is not monotone increasing as well.
\end{proof}

\begin{example}
Let $X= X_{\mathcal{G}_{2, \{0\}}}$. Then $m = \liminf \{|E_X(\ell)|\}_{\ell \in \mathbb{N}} = (3n + 2|S| + 1)^2 + |S| + 2 = 84$, so for $\ell = 84$ (since $84 > 3n + 3, 84 > m-1, 84 \pmod 2 \in \{0\},$ and $85 \pmod 2 \notin \{0\}$), $|E_{X\times Y}(\ell)| > |E_{X\times Y}(\ell + 1)|$. In particular, $|E_{X\times Y}(84)| = (85)(85) = 7225$ while $|E_{X\times Y}(85)| = (84)(86) = 7224$.
\end{example}

\begin{remark}
The reader may observe that once $\ell$ is sufficiently large for the follower or extender set sequence of $X \times Y$ to decrease, these decreases will happen for exactly the same lengths $\ell$ as the decreases in the follower or extender set sequence of $X = X_{\mathcal{G}_{n,S}}$. Thus there are infinitely many lengths for which the follower or extender set sequence of $X \times Y$ decreases.
\end{remark}

\section*{acknowledgements} 

The author thanks his advisor, Dr. Ronnie Pavlov, for many helpful conversations.

\bibliography{biblio}{}
\bibliographystyle{plain}

\end{document}